\documentclass[11pt]{article}
\usepackage{epsfig}
\usepackage{amsmath,amssymb,amsthm}
\usepackage{graphicx}
\usepackage{color}
\usepackage{cite}   
\usepackage{multicol}

\newcommand{\R}{\mathbb{R}}

\newtheorem{theo}{Theorem}

\newtheorem{proposition}{Proposition}
\newtheorem{lemme}{Lemma}

\newtheorem{rem}{Remark}

\title{Joint law of the hitting time, overshoot and undershoot for a L\'evy process}

\author{Laure Coutin\footnote{IMT, University of Toulouse, France, laure.coutin@math.univ-toulouse.fr}, Waly Ngom\footnote{FST, University Cheikh Anta Diop of Dakar, S\'en\'egal, ngomwaly@gmail.com- IMT, University of Toulouse, France, waly.ngom@math.univ-toulouse.fr}}
\begin{document}
\maketitle
\begin{abstract}
\noindent
Let $(X_{t},t\geq 0)$ be a L\'evy process which is the sum of a Brownian motion with drift and a compound Poisson process. We consider the first passage time $\tau_{x} $ at a fixed
level $ x> 0 $ by $(X_{t},t\geq 0)$ , and $K_{x}:=X_{\tau_{x}}-x$ the overshoot and $ L_{x}:=x-X_{\tau_{x^{-}}}$ the undershoot. We first study the continuity of the density of
$\tau_{x} .$ Secondly, we calculate the joint law of $(\tau_{x},K_{x}, L_{x}) . $  
\end{abstract}   
~~~~~~\textbf{ keywords}: L\'evy process, jump process, hitting time, overshoot, undershoot.

\section{Introduction}
In the theory of risk in continuous time the surplus of an insurance company is modelled by a stochastic process $(X_t, t\geq 0)$. The positive real number $x$ denotes the initial surplus and $\tau_x := \inf\{t\geq 0: X_t \geq x\}$ may be interpreted as the  default time.  
This paper deals with $\tau_x$ when $X$  is  a L\'evy process, sum of  a drifted Brownian motion  and  a compound Poisson process. Our main results lead to the regularity of the density of the hitting time and to an explicit expression characterizing the joint distribution of the triplet (first hitting time, overshoot, undershoot).\\
J. Bertoin \cite{bertoin1998levy} gives a quick and concise treatment of the core theory on L\'evy processes with the minimum of technical requirements. He gives some details on subordinators, fluctuation theory, L\'evy processes with no positive jumps and stable processes.\\  
P. Tankov and R. Cont \cite{cont2004chapman} provide a self-contained overview of theoretical, numerical and empirical research on the use of L\'evy processes in financial modeling. \\
When the process $X$ has jumps, the first results are obtained by Zolotarev \cite{zolotarev1964first} and Borovkov \cite{borovkov1965first} for $X$ a spectrally negative L\'evy process. Moreover, if $X_t$ the probability density with respect to the Lebesgue measure $p(x,t)$ then the law of $\tau_x$ has the density with respect to the Lebesgue measure $f(t,x)$ such that $xf(t,x)=tp(x,t).$
R. A. Doney \cite{doney1991hitting} deals with hitting probabilities, hitting time distributions and associated quantities for L\'evy processes which have only positive jumps. He gives an explicit formula for the joint Laplace transform of the hitting time $\tau_{x} $ and the overshoot $X_{\tau_{x}}-x.$\\
When $ X $ is a stable L\'evy process, Peskir \cite{peskir2008law} obtains an explicit formula for the passage time density. Moreover, if $X$ has no negative jumps and if $S_t= \sup_{0\leq s \leq t}X_s$ is its running supremum, Bernyk et al. \cite{bernyk2008law} show that the density function $f_t$ of $S_t$ can be characterized as the unique solution to a weakly singular Voltera integral equation of the first kind.\\
In the case where $ X $ is a jump-diffusion process, with jump   size  following a double exponential law, Kou and Wang \cite{kou2003first} give the law of $\tau_{x}.$ They obtain explicit solutions of the Laplace transform of the distribution of the first passage time. Laplace transform of the joint distribution of jump-diffusion and its running maximum, $S_t=\sup_{s\leq t}X_s$, is also obtained. Finally, they give numerical examples.
\\
For a general L\'evy process, Doney and Kyprianou \cite{doney2006overshoots} and Kyprianou \cite{kyprianou2014fluctuations} give the law of the quintuplet
 $(\bar{G}_{\tau_{x}}, \tau_{x}-\bar{G}_{\tau_{x^{-}}},X_{\tau_{x}}-x,x-X_{\tau_{x^{-}}},x-\bar{X}_{\tau_{x^{-}}})$ where 
$ \bar{X}_{t}=\sup_{s\leq t}X_{s}$ and  $\bar{G}_{t}=\sup\{s<t,X_{s}=X_{t}\}.$\\
For a stable L\'evy process $X$ of index $\alpha \in (1,2)$  the L\'evy measure of which
has the density $s(x)=cx^{-\alpha-1},~ x>~0,$  R. A. Doney in \cite{doney2008note} considers the supremum 
$S_t=\sup_{s\leq t} X_s$ of $X$.
 He shows that $S_1$ behaves as  $ s(x)\sim cx^{-\alpha-1}$ as $x \rightarrow +\infty$.
\\ 
Recently, Pog{\'a}ny, Tibor K and Nadarajah in \cite{poganymadaraj} give a shorter  
 and more general proof of R. A. Doney's  previous result \cite{doney2008note}. They derive 
 the first known closed form expression for $s(x)$ and the corresponding cumulative function, then they obtain the
  order of the remainder in the asymptotic expansion of $s(x)$.
With the same model, Alexey et al. \cite{kuznetsov2014hitting}  find the Mellin transform of the first
 hitting time of the origin and give an expression for its density.
 \\
\noindent 
Here, we show that the cumulative function of the first hitting time for one L\'evy process belongs to $\mathcal{C}(\R^*_+\times\R^*_+)$  and for any $x\in\R^*_+,$ to  $\mathcal{C}(\R_+)$ and then 
 we derive the first known closed form expression which characterizes the law of
  $(\tau_{x},K_{x}, L_{x})$. 
\\   
The paper is organized as  follows:
Section \ref{model} presents the model and the aim of the paper,
 Section \ref{regularity} studies the regularity of the law of $\tau_{x},$
  Section \ref{jointlaw}  provides  the joint law  of $(\tau_{x},K_{x}, L_{x})$. 

\section{Model and Problem to solve}
 \label{model}
 On a probability space $(\Omega, \mathcal{F}, \mathbb{P} ),$
let $ X $ be a L\'evy process, right continuous with left limit (RCLL)  starting at 0. It is defined as
\begin{align}
X_t=mt+W_t+\sum_{i=1}^{N_t} Y_{i}
\end{align} 
 where $ m \in \mathbb{R},~ W$
 is a standard Brownian motion, $N$ a Poisson process with constant positive intensity 
 $\lambda$ and $(Y_{i}, i \in \mathbb{N}^{*}) $ is a sequence of independent identically distributed 
random variables with a distribution function $F_{Y} $.  We suppose that the following $\sigma$-fields 
$\sigma(N_{t}, t \geq~0), \sigma(Y_{i}, i \in~\mathbb{N}^{*} ) $ and  $ \sigma(W_{t}, t \geq~0)$ are independent. 
 We are interested  in  the first hitting time at a level $ x>0$,
\begin{equation}
 \tau_{x}:=\inf\{t\geq 0, X_{t} \geq x\}.
\end{equation}
We also consider the overshoot and the undershoot respectively defined by
\begin{equation}
 K_{x}:=X_{\tau_{x}}-x,
\end{equation}
\begin{equation}
 L_{x}:=x-X_{\tau_{x^{-}}}.
\end{equation} 
For $\tilde{X}_{t}:=mt+W_{t}$ and $\tilde{\tau}_{x}:= \inf\{t\geq 0;\tilde{X}_{t}\geq x\},$ I. Karatzas and S. E. Shreve in  \cite{karatzas2012brownian} shown that the law of $\tilde{\tau}_{x}$   is of the form  
$\tilde{f}(u,x)du+\mathbb{P}(\tilde{\tau}_{x}^{~}= \infty)\delta_{\infty}(du)$  where
\begin{align}
\label{densitycontinue}
\tilde{f}(u,x)=\frac{\mid x \mid}{\sqrt{2\pi u^{3}}}\exp[-\frac{1}{2u}(x-mu)^{2}]1_{]0,+\infty[}(u) \mbox{ and }\mathbb{P}(\tilde{\tau}_{x}=\infty)=1-e^{mx-\mid mx \mid}.
\end{align}
L. Coutin and D. Dorobantu \cite{coutin2011first} prove the existence of the density $f_{\tau_{x}}(t,x) $ of $\tau_{x} $ and show that:

\begin{equation} 
\label{fonctiontaux}
 f_{\tau_{x}}(t,x)=\left\{ \begin{array}{ll}
       \lambda\mathbb{E}(1_{\tau_{x}>t}(1-F_{Y})(x-X_{t}))+ \mathbb{E}(1_{\tau_{x} > T_{N_{t}}}\tilde{f}(t-T_{N_{t}},x-X_{T_{N_{t}}})),&\quad \forall t > 0
                   \\                
                  \frac{\lambda}{2}(2-F_{Y}(x)-F_{Y}(x_{-}))+\frac{\lambda}{4}(F_{Y}(x)-F_{Y}(x_{-})) & \mbox{ if } t=0\\ 
                 \end{array} \right.
\end{equation}
and $ \mathbb{P}(\tau_x=\infty)=0$ if and only if $m+ \lambda \mathbb{E}(Y_1)\geq 0.$\\ 
For a more general jump-diffusion process, Roynette et al. \cite{roynette2008asymptotic} show that the Laplace
transform of $(\tau_{x},K_{x}, L_{x})$ is solution of some kind of random integral equation.\\
The problem  addressed  in this paper is  studying the regularity of the density of $\tau_x$ on  $]0,+\infty[\times]0,+\infty[,$ {then at $t=0$ for a strictly positive level $x$ fixed} and  compute an expression for the  joint distribution  of the triplet $(\tau_{x},K_{x}, L_{x}).$
\section{Regularity of  $f_{\tau_{x}}$} \label{regularity}
{This section deals with the regularity. The first subsection \ref{regularitya2} treats the continuity on $]0,\infty[\times ]0,\infty[$ as well as the last one \ref{regularityat0} studies the regularity with respect to time at $0.$
\subsection{Regularity of the density $f_{\tau_x}$ on $]0,\infty[\times ]0,\infty[$}
\label{regularitya2}
Here, our goal is to prove Proposition \ref{propocontinuity} which asserts the regularity of $\tau_x$  density law on $]0,+\infty[\times]0,+\infty[$ .}
\begin{proposition} 
\label{propocontinuity}
The application defined  on  $]0,+\infty[\times]0,+\infty[$ by 
\begin{align*}
(t,x) \longrightarrow f_{\tau_{x}}(t,x)= \lambda \mathbb{E}\left({\bf 1}_{\{\tau_x>t\}}(1-F_Y)(x-X_t)\right) + 
\mathbb{E}\left(1_{\tau_{x}>T_{N_{t}}}\tilde{f}(t-T_{N_{t}}, x-X_{T_{N_{t}}})\right)
\end{align*}
is almost surely continuous.
\end{proposition}
\begin{proof}
 Let be $(t_0,x_0) \in \mathbb{R}_+^*\times \mathbb{R}_+^*.$ We denote 
\begin{align*} 
\Omega'= \left\{\omega \in \Omega \mbox{ such that } T_{N_{t_0}(\omega)}(\omega)\neq t_0,
~~\tau_{x_0}(\omega) \neq t_0,~~x_0-X_{t_0}(\omega) \notin D_{F_Y}\right\}
\end{align*} 
where $D_{F_Y}$ is the set of the points of discontinuity of the distribution function $F_Y.$ 
\\
We assert that ${\mathbb P}(\Omega')=1$: Indeed, we have
\begin{align*}
1-{\mathbb P}(\Omega') \leq \mathbb{P}(\tau_x =0)+\mathbb{P}(X_{t_0}=x_0)+\mathbb{P}(T_{N_{t_0}}=t_0)+
  \mathbb{P}(x_0-X_{t_0} \in D_{F_Y}).
\end{align*} 
Since $\tau_{x_0},~~T_{N_{t_0}}$ and $X_{t_0}$ have a densities with respect to the Lebesgue measure and $D_{F_Y}$ is almost countable, it follows
 $$\mathbb{P}(\tau_x =0)=\mathbb{P}(X_{t_0}=x_0)=\mathbb{P}(T_{N_{t_0}}=t_0)= \mathbb{P}(x_0-X_{t_0} \in D_{F_Y})=0.$$ 

Note that ${\bf 1}_{\{\tau_x> T_{N_t}\}}={\bf 1}_{\{X_{T_{N_t}}^*<x\}}$ where $X_t^*= \sup_{u\leq t}X_u.$

The random variable $X_{T_{N_{t_0}}}^*$ is reached by the process $X$ either before $T_{N_{t_0}}$ either at $T_{N_{t_0}}.$\\
Let $\omega $ be fixed in $\Omega'.$
\\  
$\bullet$
 On the  event  ${\{\tau_{x_0}> T_{N_{t_0}}\}}={\{X_{T_{N_{t_0}}}^*<x_0\}},$
$X_{T_{N_{t_0}(\omega)}(\omega)}^*(\omega)= X_{T_{N_{t_0}(\omega)}(\omega)}(\omega)\neq~x_0.$
\\
Thus, if  $X_{T_{N_{t_0}(\omega)}(\omega)}^*(\omega)$ is reached at $T_{N_{t_0}(\omega)}(\omega),$ either it is less than $x_0$ or more than $x_0.$\\
If $X_{T_{N_{t_0}(\omega)}(\omega)}^*(\omega)= X_{T_{N_{t_0}(\omega)}(\omega)}(\omega)< x_0,$ since $T_{N_{t_0}(\omega)}(\omega) <~t_0 <~ T_{N_{t_0}(\omega)+1}(\omega),$  there exists  $\varepsilon_0(\omega)>0$ and $\delta_0(\omega)>0$ such that for any $t$ satisfying $|t-~t_0|~\leq~\delta_0(\omega)$  we have $x_0-\varepsilon_0(\omega)<X_{T_{N_{t}(\omega)}(\omega)}^*(\omega)< x_0+\varepsilon_0(\omega).$\\
That means for $(t,x)$ such that $|t-t_0|< \delta_0(\omega)$ and $|x-x_0|<\varepsilon_0(\omega),$ we have 
$${\bf 1}_{\{\tau_x(\omega)> T_{N_t(\omega)}(\omega)\}}={\bf 1}_{\{\tau_{x_0}(\omega)> T_{N_{t_0}(\omega)}(\omega)\}}=1 .$$
If $X_{T_{N_{t_0}(\omega)}(\omega)}^*(\omega)= X_{T_{N_{t_0}(\omega)}(\omega)}(\omega)> x_0,$  since $T_{N_{t_0}(\omega)}(\omega) <~t_0 <~ T_{N_{t_0}(\omega)+1}(\omega),$  there exists $\varepsilon_1(\omega)>0$ and $\delta_1(\omega)>0$ such that for any $t$ satisfying $|t-~t_0|~\leq~\delta_1(\omega)$  we have $x_0-\varepsilon_0(\omega)<X_{T_{N_{t}}}^*< x_0+\varepsilon_0(\omega).$\\
That means for $|t-t_0|<\delta_1(\omega)$ and $|x-x_0|<\varepsilon_1(\omega),$
$$ {\bf 1}_{\{\tau_x(\omega)> T_{N_t(\omega)}(\omega)\}}={\bf 1}_{\{\tau_{x_0}(\omega)> T_{N_{t_0}(\omega)}(\omega)\}}=0.$$
In the two above cases, we conclude that for any $\omega \in \Omega,$ there exists $\delta(\omega)  >0$ and $\varepsilon(\omega)  >0$ such that:
\begin{align}\label{valmax=valextreme}
|t-t_0|< \delta(\omega)  \mbox{ and } |x-x_0|<\varepsilon(\omega) \mbox{ imply that } {\bf 1}_{\{\tau_x(\omega)> T_{N_t(\omega)}(\omega)\}}={\bf 1}_{\{\tau_{x_0}(\omega)> T_{N_{t_0}(\omega)}(\omega)\}}
\end{align}
\\
$\bullet$
If $X_{T_{N_{t_0}(\omega)}(\omega)}^*(\omega)$ is reached at $v<T_{N_{t_0}(\omega)}(\omega),$ either it is less than $x_0$ or more than $x_0.$

If  $X_{T_{N_{t_0}(\omega)}(\omega)}^*(\omega)= X_v(\omega)<x_0,$ since $T_{N_{t_0}(\omega)}(\omega) <~t_0 <~ T_{N_{t_0}(\omega)+1}(\omega),$  there exists  $\varepsilon_4(\omega)>0$ and $\delta_4(\omega)>0$ such that for any $t$ satisfying $|t-~t_0|~\leq~\delta_4(\omega)$  we have $x_0-\varepsilon_4(\omega)<X_{T_{N_{t}(\omega)}}^*(\omega)< x_0+\varepsilon_4(\omega).$\\
That means for $(t,x)$ such that $|t-t_0|< \delta_4(\omega)$ and $|x-x_0|<\varepsilon_4(\omega),$ we have 
$${\bf 1}_{\{\tau_x> T_{N_t}\}}={\bf 1}_{\{\tau_{x_0}> T_{N_{t_0}}\}}=1 .$$

If  $X_{T_{N_{t_0}(\omega)}(\omega)}^*(\omega)= X_v(\omega)> x_0,$ since $T_{N_{t_0}(\omega)}(\omega) <~t_0 <~ T_{N_{t_0}(\omega)+1}(\omega),$  there exists $\varepsilon_5(\omega)>0$ and $\delta_5(\omega)>0$ such that for any $t$ satisfying $|t-~t_0|~\leq~\delta_5(\omega)$  we have $x_0-\varepsilon_5(\omega)<X_{T_{N_{t}}}^*< x_0+\varepsilon_5(\omega).$
That means for $|t-t_0|<\delta_5(\omega)$ and $|x-x_0|<\varepsilon_5(\omega),$
$$ {\bf 1}_{\{\tau_x> T_{N_t}\}}={\bf 1}_{\{\tau_{x_0}> T_{N_{t_0}}\}}=0.$$

If  $X_{T_{N_{t_0}(\omega)}(\omega)}^*(\omega)= X_v(\omega)= x_0,$
we consider a function which is equal to $\tilde{f}$ on $]0,+\infty[\times]0,+\infty[$ and $0$ on $]0,+\infty[\times \mathbb{R}_-.$  We denote it $\tilde{f}$ again.
This function is everywhere continuous.
 Since $T_{N_{t_0}(\omega)}(\omega) <~t_0 <~ T_{N_{t_0}(\omega)+1}(\omega),$ and  $(t,x) \longrightarrow\tilde{f}(t-T_{N_t(\omega)}(\omega),x-X_{T_{N_t(\omega)}(\omega)}(\omega))$ is continuous at $(t_0,x_0),$ 
  there exists $\varepsilon_6(\omega)>0$ and $\delta_6(\omega)>0$ such that for any $(t,x)$ satisfying $|t-~t_0|~\leq~\delta_6(\omega)$ and $|x-x_0|\leq \varepsilon_6(\omega)$, we have
\begin{align}
\label{valeur=valinterm=}
&\lim_{(t,x) \rightarrow (t_0,x_0)}{\bf 1}_{\{X_{T_{N_t(\omega)}(\omega)}^*(\omega)<x\}}\tilde{f}(t-T_{N_t(\omega)}(\omega),x-X_{T_{N_t(\omega)}(\omega)}(\omega))\\
&={\bf 1}_{\{X_{T_{N_{t_0}(\omega)}(\omega)}^*<x_0\}}\tilde{f}(t_0-T_{N_{t_0(\omega)}(\omega)},x_0-X_{T_{N_{t_0}(\omega)}(\omega)}(\omega))=0.
\nonumber
\end{align}
Using uniform integrability of the family 
$\left(1_{\tau_{x}>T_{N_{t}}}\tilde{f}(t-T_{N_{t}}, x-X_{T_{N_{t}}}),~~t>0,~~x>0 \right),$ obtained from Lemma \ref{ui},  the continuity of 
$$(t,x) \longrightarrow \mathbb{E}\left(1_{\tau_{x}>T_{N_{t}}}\tilde{f}(t-T_{N_{t}}, x-X_{T_{N_{t}}})\right)$$
at $(t_0,x_0)$ follows.\\
 
Since the family $\left(({\bf 1}_{\{\tau_x>t\}}(1-F_Y)(x-X_t), t>0,~ x>0\right)$ is bounded by $1$, it is  uniformly integrable and we proceed analogously to the previous to obtain the continuity at $(t_0,x_0)$
of 
$$(t,x) \longrightarrow \mathbb{E}\left({\bf 1}_{\{\tau_x>t\}}(1-F_Y)(x-X_t)\right).$$

\end{proof}

 We now study the regularity  with respect to time at $0$.
\subsection{Regularity of  $f_{\tau_{x}}$ with respect to time at $0$}\label{regularityat0}
{The next two propositions show that for any fixed $x>0,$ the  $f_{\tau_{x}}$ density law is continuous with respect to time at $0.$ }
\begin{proposition} 
\label{prop1b}
Let be  $ x > 0 $ fixed, we have
\begin{align*}
\lim_{t \mapsto 0}  \mathbb{E}\left( {\bf 1}_{\tau_x>t}[1-F_Y](x-X_{t})\right)=
\frac{1}{2×}\left(2- F_{Y}(x)-F_{Y}(x_{-})\right).
\end{align*}
\end{proposition}
\begin{proof}
We have
\begin{align*}
  \mathbb{E}( {\bf 1}_{\tau_x>t}[1-F_Y](x-X_{t}))= \mathbb{E}({\bf 1}_{\{N_t=0\}} {\bf 1}_{\tau_x>t}[1-F_Y](x-X_{t}))+
   \mathbb{E}({\bf 1}_{\{N_t>0\}} {\bf 1}_{\tau_x>t}[1-F_Y](x-X_{t})).
\end{align*}
But 
\begin{align*}
(i)& ~~0 \leq \lim_{t \mapsto 0}  \mathbb{E}({\bf 1}_{\{N_t>0\}} {\bf 1}_{\tau_x>t}[1-F_Y](x-X_{t})) \leq 
\lim_{t \mapsto 0} \mathbb{P}(N_t \geq 1)= \lim_{t \mapsto 0} 1- e^{-a t} = 0
\\
(ii)&~~\mathbb{E}\left({\bf 1}_{\{N_t=0\}} {\bf 1}_{\tau_x>t}[1-F_Y](x-X_{t})\right)=e^{-at}\mathbb{E}\left({\bf 1}_{\{\tilde{X}_t^*<x\}}[1-F_Y](x-\tilde{X}_t)\right)
\end{align*}

 We first remark that 
 \small{
$$|\mathbb{E}\left({\bf 1}_{\{\tilde{X}_t^*<x\}}[1-F_Y](x-\tilde{X}_t)\right)-\mathbb{E}\left({\bf 1}_{\{\tilde{X}_t<x\}}[1-F_Y](x-\tilde{X}_t)\right)|\leq |\mathbb{E}({\bf 1}_{\{\tilde{X}_t^*<x\}}-{\bf 1}_{\{\tilde{X}_t<x\}})|=\mathbb{P}(\{\tilde{X}_t<x\leq \tilde{X}_t^*\}).$$  }
The density function of  $(\tilde{X}_t^*,\tilde{X}_t)$ given by Corollary 3.2.1.2  p. 147  \cite{jeanblanc2009mathematical} yields
$$\mathbb{P}(\{\tilde{X}_t<x\leq \tilde{X}_t^*\})=
\int_x^\infty db\int_{-\infty}^b da \frac{2(2b-a)}{\sqrt{2\pi t^3}} \exp\left[-\frac{(2b-a)^2}{2t}+ma-\frac{m^2}{2}t\right].$$
This integral is bounded (with respect to a multiplicative constant $C$) by
$$\mathbb{P}(\{\tilde{X}_t<x\leq \tilde{X}_t^*\})\leq C \int_x^\infty db\int_{-\infty}^b da \frac{(2b-a)}{\sqrt{t^3}} 
\exp\left[-\frac{(2b-a)^2}{2t}\right].$$
Notice that the application $t\to \frac{(2b-a)}{\sqrt{t^3}} \exp\left[-\frac{(2b-a)^2}{2t}\right]$
is decreasing to $0$ when $t\downarrow 0.$ So  Lebesgue's monotonous convergence theorem proves that
\begin{align}\label{fluct}
\lim_{t\to 0} \mathbb{P}(\{\tilde{X}_t<x\leq \tilde{X}_t^*\})=0.
\end{align}
Secondly,  
\begin{align*}
\mathbb{E}\left({\bf 1}_{\{\tilde{X}_t<x\}}[1-F_Y](x-\tilde{X}_t)\right)= \mathbb{E}\left({\bf 1}_{\{\tilde{X}_t\leq 0\}}[1-F_Y](x-\tilde{X}_t)\right)+
\mathbb{E}\left({\bf 1}_{\{0\leq \tilde{X}_t<x\}}[1-F_Y](x-\tilde{X}_t)\right).
\end{align*}
Since $F_Y$ is bounded and RCLL and   $\tilde{X}$ continuous, Lebesgue's dominated convergence theorem yields
\begin{align*}
\lim_{t \mapsto 0} \mathbb{E}\left({\bf 1}_{\{N_t=0\}} {\bf 1}_{\tau_x>t}[1-F_Y](x-X_{t})\right)=
\frac{1}{2×}\left(2- F_{Y}(x)-F_{Y}(x_{-})\right).
\end{align*}
\end{proof}

\begin{proposition}
\label{prop-A-1}
Let be  $ x > 0 $ fixed, we have
\begin{align*}
\lim_{t\rightarrow 0}{\mathbb E} \left( {\mathbf 1}_{\{\tau_x >T_{N_t}\}}\tilde{f}(t-T_{N_t},x-X_{T_{N_t}})\right)=\frac{\lambda}{4} \left(F_Y(x) -F_Y(x_-)\right).
\end{align*}
\end{proposition}

\begin{proof}
(i) We first deal with 
$$\lim_{t\to 0}{\mathbb E} \left( {\mathbf 1}_{\{N_t=0,\tau_x >T_{N_t}\}}\tilde{f}(t-T_{N_t},x-X_{T_{N_t}})\right)=
lim_{t\to 0}{\mathbb E} ( {\mathbf 1}_{\{N_t=0\}})\tilde{f}(t,x))=lim_{t\to 0}(1-e^{-at})\tilde{f}(t,x)=0$$
using Definition (\ref{densitycontinue}).
 
(ii) Then we deal with 
$${\mathbb E} \left( {\mathbf 1}_{\{N_t\geq 2,\tau_x >T_{N_t}\}}\tilde{f}(t-T_{N_t},x-X_{T_{N_t}})\right).$$
We use Lemma A.1 of \cite{coutin2011first} for $p=1,$  law of $G$ being the Gaussian law ${\cal N}(0,1)$,
then
\begin{align*}
{\mathbb E} \left(\tilde{f}(u ,\mu+ \sigma G){\mathbf 1}_{\{ \mu + \sigma G>0\}} \right)=
\frac{1}{\sqrt{2 \pi}} \frac{e^{- \frac{(\mu - mu)^2}{2(\sigma^2 +u)}}}{\sqrt u(\sigma^2 +u)}
{\mathbb E} \left[ \left( \sigma G + \sqrt{\frac{u}{\sigma^2 +u}} (\mu- mu) + m \sqrt{u(\sigma^2 +u)}\right)_+ \right].
\end{align*}
Using $C_{1/2}= \sup \sqrt{y} e^{-\frac{y^2}{2}},$ $y=\frac{\mu-mu}{\sqrt{\sigma^2+u}},$ 
\begin{align*}
{\mathbb E} \left(\tilde{f}(u,\mu +\sigma G){\mathbf 1}_{\{ \mu + \sigma G>0\}} \right)\leq 
\frac{1}{\sqrt{2 \pi}} \left[ \frac{1}{\sqrt{u}(\sigma^2 +u)}
 \left( {\mathbb E} |\sigma G|+ \sqrt{u} C_{1/2} + | m| \sqrt{u(\sigma^2 +u)}\right) \right].
\end{align*}
For $u=t- T_{N_t},$ $\sigma^2= T_{N_t}$ and $u+\sigma^2=t,$ we obtain using the independence between the Poisson process and the Brownian motion 
\begin{align*}
{\mathbb E} \left( {\mathbf 1}_{\{N_t \geq 2\} }{\mathbf 1}_{\{\tau_x > T_{N_t}\}}
\tilde{f}( t-T_{N_{t}}, x- X_{T_{N_t}}) \right) \leq 
{\mathbb E} \left( {\mathbf 1}_{\{ N_t \geq 2\}}\frac{1}{t \sqrt{2 \pi}}
\left[ \frac{E[|\sigma G|]}{\sqrt{t- T_{N_t}}} + C_{1/2} + |m|\sqrt t \right] \right).
\end{align*}
So, using ${\mathbb P}(\{ N_t \geq 2\})=0(t^2)$ and  Lemma \ref{lem-A-1},  
\begin{align*}
\lim_{t \rightarrow 0} {\mathbb E} \left( {\mathbf 1}_{\{N_t \geq 2\} }{\mathbf 1}_{\{\tau_x > T_{N_t}\}}
\tilde{f}( t-T_{N_{t}}, x- X_{T_{N_t}}) \right)=0.
\end{align*}
(iii) Finally we deal  with
$$
A_t= {\mathbb E} \left( {\mathbf 1}_{\{N_t=1,\tau_x >T_{N_t}\}}\tilde{f}(t-T_{N_t},x-X_{T_{N_t}})\right)=
{\mathbb E} \left( {\mathbf 1}_{\{N_t=1,\tau_x >T_1\}}\tilde{f}(t-T_1,x-X_{T_1})\right).
$$
Since  the event $\{N_t=1,\tau_x >T_1\}= \{T_1\leq t <T_2,\tilde{X}^*_{T_1} <x,\tilde{X}_{T_1} + Y_1 <x\}$ we have 
\begin{align*}
A_t={\mathbb E} \left( {\mathbf 1}_{\{T_1\leq t <T_2,\tilde{X}^*_{T_1} <x,\tilde{X}_{T_1} + Y_1 <x\}}\tilde{f}(t-T_{1},x-\tilde{X}_{T_1}-Y_1)\right).
\end{align*}
 Using the law of $T_2-T_1$ and its independence  from $T_1, \tilde{X}_{T_1}^*,\tilde{X}_{T_1}, Y_1$, it follows that 
\begin{align*}
A_t={\mathbb E} \left( e^{-a( t-T_1)}{\mathbf 1}_{\{T_1\leq t ,\tilde{X}^*_{T_1} <x, \tilde{X}_{T_1} + Y_1 <x\}}\tilde{f}(t-T_{1},x-\tilde{X}_{T_1}-Y_1)\right)
\end{align*}
Using the law of  $T_1$ and the independence between $T_1$ and $(\tilde{X},Y)$, yields
\begin{align*}
A_t=ae^{-at} \int_0^t du {\mathbb E} \left( {\mathbf 1}_{\{ \tilde{X}^*_u <x, \tilde{X}_u + Y_1 <x\}}\tilde{f}(t-u,x-\tilde{X}_{u}-Y_1)\right).
\end{align*}
Since $\tilde{X}_u$ and $ Y_1$ are independent, conditioning by  $(\tilde{X}_u, Y_1)$ and using Lemma \ref{lem-A-3} for  $c= x,$   
\begin{align*}
A_t=ae^{-at} \int_0^t du {\mathbb E} \left( {\mathbf 1}_{\{\tilde{X}_u  <\min (x, x-Y_1) \}}
\left[ 1 - e^{-\frac{2x^2 -2 x\tilde{X}_u}{u}}\right] f(t-u,x-\tilde{X}_{u}-Y_1)\right).
\end{align*}
The change of  variable $u=ts$ leads to
\begin{align*}
A_t=ae^{-at} \int_0^1 t d s{\mathbb E} \left( {\mathbf 1}_{\{  \tilde{X}_{st}  <\min (x, x-Y_1) \}}
\left[ 1 - e^{-\frac{2x^2 -2 x\tilde{X}_{ts}}{ts}}\right] f(t(1-s),x-\tilde{X}_{st}-Y_1)\right).
\end{align*}
The density of  $\tilde{X}_{ts}~: $ 
$\frac{1}{\sqrt{2\pi ts}} e^{- \frac{(g-mts)^2}{2ts}}$
and  $\tilde{f}$ defined in (\ref{densitycontinue})  yield
\begin{align*}
&A_t=ae^{-at}
 \int_0^1 t ds 
 \\
& {\mathbb E} \left(\int_{-\infty}^{\min (x, x-Y_1)} 
\left[ 1 - e^{-\frac{2x^2 -2 xg}{ts}}\right] \frac{ x-Y_1 - g}{\sqrt{2\pi t^3(1-s)^3}} e^{- \frac{(x-g-Y_1 - mt(1-s))^2}{2t(1-s} } \frac{1}{\sqrt{2 \pi ts}}e^{- \frac{(g-mts)^2}{2ts}} dg\right).
\end{align*}

Let be  
$z=x-Y_1 -mt(1-s),$ $y= mts,$ $u = t(1-s)$ and  $v=ts,$
with
\begin{align*}
&\frac{vz+uy}{u+v} = \frac{ts[x-Y_1 -mt(1-s)]+ t(1-s)mts}{t}=s(x-Y_1),\\
&z-y= x-Y_1 -mt(1-s) - mts= x-Y_1-mt.
\end{align*}
 By Lemma \ref{lem-A-0}, 
\begin{align*}
&A_t=ae^{-at} \int_0^1t ds \\
& {\mathbb E} \left(\int_{-\infty}^{\min (x, x-Y_1)}
\left[ 1 - e^{-\frac{2x^2 -2 xg}{ts}}\right] \frac{ x-Y_1 - g}{\sqrt{(2\pi)^2 t^4(1-s)^3s}} e^{- \frac{(g-s(x-Y_1))^2}{2ts(1-s)} - \frac{(x-Y_1- mt)^2}{2t^2s(1-s)}}dg \right). 
\end{align*}
A new change of  variable $g'= \frac{g-s(x-Y_1)}{\sqrt{ts(1-s)}}$ meaning  $ g= \sqrt{ts(1-s)} g' + s(x-Y_1),$
 and $x-Y_1-g= (x-Y_1)(1-s) - \sqrt{ts(1-s)}g'$ implies
 \small{
\begin{align*}
&A_t=ae^{-at} \int_0^1 ds \int dg'\sqrt{t^3s(1-s)} \\
& {\mathbb E} \left({\mathbf 1}_{\{   g'   <\frac{\min (x, x-Y_1) -s(x-Y_1)}{\sqrt{ts(1-s)} } \}}
\left[ 1 - e^{-\frac{2x^2 -2 x(\sqrt{ts(1-s)} g '+ s(x-Y_1))}{ts}}\right] \frac{ (x-Y_1 )(1-s) - \sqrt{ts(1-s)}g'}{\sqrt{(2\pi)^2 t^4(1-s)^3 s}} e^{-\frac{( g')^2}{2}- \frac{(x-Y_1-mt)^2}{2t^2s(1-s)}} \right).\\
\end{align*} }
This would mean
\begin{align*}
&A_t=ae^{-at} \int_0^1 ds \int dg'\\ 
 &{\mathbb E} \left[ {\mathbf 1}_{\{   g'   <\frac{\min (x, x-Y_1) -s(x-Y_1)}{\sqrt{ts(1-s)} } \}}
\left[ 1 - e^{-\frac{2x^2 -2 x(\sqrt{ts(1-s)} g '+ s(x-Y_1))}{ts}}\right] e^{-\frac{( g')^2}{2}- \frac{(x-Y_1-mt)^2}{2t^2s(1-s)}}
\left[ \frac{ (x-Y_1 )  }{2\pi \sqrt{t }}-  \frac{g'\sqrt{s}}{2\pi \sqrt{1-s}} \right]\right].
\end{align*}
 On the set 
 $\{g' <c\},$ $2x^2 - 2xc >0$. This would mean 
 $\{g'   <\frac{\min (x, x-Y_1 -s(x-Y_1)}{\sqrt{ts(1-s)} }) \},$  
 where $c=\frac{x-Y_1 -s(x-Y_1)}{\sqrt{ts(1-s)} }$  
 $ -\frac{2x^2 -2 x(\sqrt{ts(1-s)} g '+ s(x-Y_1))}{ts} <0,$ 
and
\begin{align*}
\left| 1 - e^{-\frac{2x^2 -2 x(\sqrt{ts(1-s)} g '+ s(x-Y_1))}{t}}\right|{\mathbf 1}_{\{   g'   <\frac{\min (x, x-Y_1) -s(x-Y_1)}{\sqrt{ts(1-s)} } \}} \leq 1;
\end{align*}
with
\begin{align*}
\lim_{t\rightarrow 0} \left[ 1 - e^{-\frac{2x^2 -2 x(\sqrt{ts(1-s)} g '+ s(x-Y_1))}{t}}\right|]
{\mathbf 1}_{\{ g'   <\frac{\min (x, x-Y_1) -s(x-Y_1)}{\sqrt{ts(1-s)} } \}} =1.
\end{align*}
 
Morever,
\begin{align*}
\frac{\left| x-Y_1\right| }{\sqrt{t}}e^{- \frac{(x-Y_1-mt)^2}{2t^2s(1-s)}}&\leq \frac{|mt| + \left| x-Y_1-mt \right| }{\sqrt{t}}e^{- \frac{(x-Y_1-mt)^2}{2t^2s(1-s)}}
\\  
 & \leq \frac{ |mt| + C_{1/2}\sqrt{t^2s(1-s)}}{\sqrt{t}}=
 \sqrt t\left(|m|+ C_{1/2}\sqrt{s(1-s)}\right)
\end{align*}
Thus
$$
\lim_{t\rightarrow 0} \frac{\left| x-Y_1\right| }{\sqrt{t}}e^{- \frac{(x-Y_1-mt)^2}{2t^2s(1-s)}}=0.
$$

Finally,
\begin{align*}
e^{-\frac{( g')^2}{2}- \frac{(x-Y_1-mt)^2}{2t^2s(1-s)}}\frac{|g'|\sqrt{s}}{2\pi \sqrt{1-s}}\leq e^{-\frac{( g')^2}{2}}\frac{|g'|\sqrt{s}}{2\pi \sqrt{1-s}}
\end{align*}
and

\begin{align*}
\lim_{t\rightarrow 0}  {\mathbf 1}_{\{   g'   <\frac{\min (x, x-Y_1) -s(x-Y_1)}{\sqrt{ts(1-s)} } \}}e^{-\frac{( g')^2}{2}- \frac{(x-Y_1-mt)^2}{2t^2s(1-s)}}\frac{g'\sqrt{s}}{2\pi \sqrt{1-s}}=
 e^{-\frac{( g')^2}{2}}\frac{g'\sqrt{s}}{2\pi \sqrt{1-s}}{\mathbf 1}_{\{Y_1 = 0\}} {\mathbf 1}_{\{g'<0\}}.
\end{align*}
By Lebesgue's dominated convergence theorem ,

\begin{align*}
\lim_{t\rightarrow 0} A_t = -a {\mathbb P}(Y_1=x)\int_0^1 ds \int_{\{g'<0\}} dg'e^{-\frac{( g')^2}{2}}  \frac{g'\sqrt{s} }{2 \pi \sqrt{1-s}} 
\end{align*}
and $\lim_{t \rightarrow 0} A_t=a{\mathbb P}(Y_1=x) \int_0^1 \frac{\sqrt{s} }{2 \pi \sqrt{1-s}} ds=\frac{a{\mathbb P}(Y=x)}{4}.$
Indeed, $\beta(3/2,1/2)= \frac{1}{2} \gamma(1/2)^2$ and $\gamma(1/2)=\pi.$
\end{proof}

\section{The joint law}\label{jointlaw}
If the default time coincides with a jump time of the process $X$, it is also  important to have information on the deficit, namely overshoot, right after the default and on the surplus, namely undershoot of the firm, immediately before the default time. Therefore, A. Volpi et al. \cite{roynette2008asymptotic} 
deal with the asymptotic behavior of the triplet $(\tau_x, K_x, L_x)$    by showing after a renormalization of $\tau_x$ that it converges in distribution as $x$ goes to $\infty$.
 To characterize the joint law  of $(\tau_x, K_x, L_x)$ on $\mathbb{R}_+ \times \mathbb{R}_+ \times \mathbb{R}_+$ for any $x>0$, this section's  main result is
  the next theorem. From now on, we consider two continuous bounded functions $\Phi$ and $\Psi.$ The methodology used for the proof is inspired from  Coutin-Dorobantu \cite{coutin2011first} who study the law of $\tau_x$ which is here the first marginal distribution. It consists in splitting $\mathbb{E}(1_{t<\tau_{x}\leq t+h}\Phi (K_{x})\Psi(L_{x}))$ following the values of $N_{t+h}-N_t$ for $t=0$ then for $t>0.$ 
\begin{theo} \label{th1} 
The joint law of the triplet $(\tau_x, K_x, L_x)$, conditionally on $\{ \tau_x < \infty\}$,  is given on $\mathbb{R}_+ \times \mathbb{R}_+ \times \mathbb{R}_+$ by $ p(.,.,.,x)$ such that:
\begin{align*}
p(0,dk,dl)&= \frac{\lambda}{4}[F_Y(x)-F_Y(x_-)]\delta_{\{0,0,0\}}(dt,dk,dl)+ \lambda F_l(dk)\delta_{\{0,x\}}(dt,dl)\\ &+\frac{\lambda}{2}\Delta F_Y(x)\delta_{\{0,0,x\}}(dt,dk,dl)
\end{align*}
and for every $t>0$, 
 \begin{align*}
 p(dt,dk,dl)&=\mathbb{E}[{\bf 1}_{\{\tau_{x}>T_{N_{t}}\}}\tilde{f}(t-T_{N_{t}},x-X_{T_{N_{t}}})]\delta_{\{0,0\}}(dk,dl)dt\\
&+\lambda \mathbb{E}\left[{\bf 1}_{\{k\geq 0, l\geq 0\}} {\bf 1}_{\{\tau_{x}>T_{N_{t}}\}}f_0(x-X_{T_{N_{t}}}-l)\right]F_l(dk)dldt\\
&- \lambda\mathbb{E}\left[{\bf 1}_{\{k\geq 0,l\geq 0\}}{\bf 1}_{\{\tau_{x}>T_{N_{t}}\}}f_0(X_{T_{N_{t}}}-x-l)
\exp(2m(x-X_{T_{N_{t}}}))\right]F_l(dk) dldt\\
\end{align*}
where
$f_0$ is the density function of a Gaussian random variable with mean $\mu= m(t-T_{N_{t}})$ and variance 
$\sigma^2 = t-T_{N_{t}}, ~~ \tilde{f}$ is defined by (\ref{densitycontinue}),
 $F_l(dk)$ is the image of $F_Y(dk)$ by the map $y\mapsto y-l$ and $ \Delta F_Y(x)=~F_Y(x)~-~F_Y(x_-)$.
\end{theo}

\begin{rem}
Referring to \cite{roynette2008asymptotic}, for all $x >0$, the first passage  time $\tau_x $ is finite almost surely if and only if
 $ m+ \mathbb{E}(Y_{1}) \geq 0.$
\end{rem}
To prove the theorem, we use Propositions \ref{cazero} and \ref{casaut}. Indeed, Proposition \ref{cazero} gives 
the law at time $t=0$ and Proposition \ref{casaut} deals with time $t>0$. The proof of Proposition \ref{casaut} is broken in three parts: Step 1, Step 2 and Proposition  \ref{loijointe}.

\begin{proposition}\label{cazero}
\begin{align*}
\lim_{h\rightarrow 0} \frac{1}{h} \mathbb{E}(1_{\tau_{x}\leq h}\Phi (K_{x})\Psi(L_{x}))&= \Phi(0)\Psi(0)\frac{\lambda}{4}[F_Y(x)-F_Y(x_-)]+\lambda \mathbb{E}[\Phi(Y_1-x)\Psi(x){\bf 1}_{\{Y_1>x\}}]\\
&+ \frac{\lambda}{2}\mathbb{E}[\Phi(0)\Psi(Y_1){\bf 1}_{\{Y_1=x\}}].
\end{align*}
\end{proposition}
\begin{proof} 
We split $\mathbb{E}(1_{\tau_{x}\leq h}\Phi (K_{x})\Psi(L_{x}))$ according to the values of $N_h$:
\begin{align*}
\mathbb{E}(1_{\tau_{x}\leq h}\Phi (K_{x})\Psi(L_{x}))&= \mathbb{E}(1_{\tau_{x}\leq h}{\bf 1}_{\{N_h=0\}}\Phi (K_{x})\Psi(L_{x}))+ \mathbb{E}(1_{\tau_{x}\leq h}{\bf 1}_{\{N_h=1\}}\Phi (K_{x})\Psi(L_{x})\\
&+ \mathbb{E}(1_{\tau_{x}\leq h}{\bf 1}_{\{N_h\geq 2\}}\Phi (K_{x})\Psi(L_{x}).
\end{align*}
By hypothesis, $\Phi$ and $\Psi$ are bounded  and on the event $\{ N_h=0\}$, the law of $\tau_x$ is the one of 
 $\tilde{\tau}_x$, so has  the continuous density $\tilde{f}(.,x)$ defined in (\ref{densitycontinue}). Since
\begin{align*}
\mathbb{E}(1_{\tau_{x}\leq h}{\bf 1}_{\{N_h=0\}}\Phi (K_{x})\Psi(L_{x}))= \Phi (0)\Psi(0)\mathbb{P}(\tilde{\tau}_x \leq h)
\end{align*}
and
\begin{align*}
| \mathbb{E}(1_{\tau_{x}\leq h}{\bf 1}_{\{N_h\geq 2\}}\Phi (K_{x})\Psi(L_{x})|\leq ||\Phi||_\infty ||\Psi||_\infty \mathbb{P}(N_h\geq 2),
\end{align*}
it follows that     
\begin{align}\label{sansaut}
\lim_{h \rightarrow 0} \frac{1}{h}\mathbb{E}(1_{\tau_{x}\leq h}{\bf 1}_{\{N_h=0\}}\Phi (K_{x})\Psi(L_{x}))=0
\end{align}
and
\begin{align*}
\lim_{h \rightarrow 0} \frac{1}{h}\mathbb{E}(1_{\tau_{x}\leq h}{\bf 1}_{\{N_h\geq 2\}}\Phi (K_{x})\Psi(L_{x})=0.
\end{align*}
It remains to study $\mathbb{E}(1_{\tau_{x}\leq h}{\bf 1}_{\{N_h=1\}}\Phi (K_{x})\Psi(L_{x})$. For this purpose, we split it as in Coutin and Dorobantu \cite{coutin2011first} according to the relative positions of $\tau_x$ and $T_1$ the first jump time of the process $N$.
\small{ 
\begin{align*}
\mathbb{E}(1_{\tau_{x}\leq h}{\bf 1}_{\{N_h=1\}}\Phi (K_{x})\Psi(L_{x}))&=\mathbb{E}(1_{\tau_{x}\leq h}{\bf 1}_{\{N_h=1\}}{\bf 1}_{\{\tau_x<T_1\}}\Phi (K_{x})\Psi(L_{x}))\\&+ \mathbb{E}(1_{\tau_{x}\leq h}{\bf 1}_{\{N_h=1\}}{\bf 1}_{\{\tau_x=T_1\}}\Phi (K_{x})\Psi(L_{x}))\\ &+\mathbb{E}(1_{\tau_{x}\leq h}{\bf 1}_{\{N_h=1\}}{\bf 1}_{\{\tau_x>T_1\}}\Phi (K_{x})\Psi(L_{x}))\\
&= A_1(h)+ A_2(h)+A_3(h).
\end{align*}  }
{ On the set $\{ \tau_x\leq h, \tau_x\neq T_1, N_h=1\}$, the process $X$ is continuous at $\tau_x$. and $K_x=L_x=0$ . Therefore, Step 1 and Step 3 of Subsection 2.1 in \cite{coutin2011first} imply that:  
\begin{align*}
\lim_{h \rightarrow 0} \frac{1}{h} A_1(h)=0
\end{align*}
and
\begin{align*}
\lim_{h \rightarrow 0} \frac{1}{h} \mathbb{E}(1_{\tau_{x}\leq h}{\bf 1}_{\{N_h=1\}}{\bf 1}_{\{\tau_x>T_1\}})
=  \frac{\lambda}{4}[F_Y(x)-F_Y(x_-)].
\end{align*}
}
To study $A_2(h)$, we observe that:
 \begin{align*}
&\tau_x= T_1 \mbox{ if and only if } \tilde{X}^*_{T_1}< x \mbox{ and } \tilde{X}_{T_1}+ Y_1> x\\
&\mbox{ and on this set } K_x= \tilde{X}_{T_1}+Y_1-x,~ 
L_x= x-\tilde{X}_{T_1}\mbox{ and } Y_1 > 0.
\end{align*}

Therefore ,
\begin{align*}
A_2(h)= \mathbb{E}\left({\bf 1}_{\{T_1 \leq h <T_2\}}{\bf 1}_{\{\tilde{X}_{T_1}^*<x,\tilde{X}_{T_1}+Y>x\}}
\Phi(\tilde{X}_{T_1}+Y-x)\Psi(x-\tilde{X}_{T_1})\right).
\end{align*}
The independence of $(S_i, i\geq 1)$ and $( Y_1, \tilde{X}, \tilde{\tau}_x)$ leads after integrating with respect to $S_2$, then $ S_1$ to 
\begin{align*}
 \frac{1}{h} A_2(h)= \frac{ \lambda e^{-\lambda h}}{h}\int_0^h
\mathbb{E}\left({\bf 1}_{\{\tilde{X}_{u}^*<x<\tilde{X}_u +Y_1\}}
\Phi(\tilde{X}_{u}+Y_1-x)\Psi(x-\tilde{X}_{u})\right)du\\
\end{align*}
\begin{align*}
|\mathbb{E}\left({\bf 1}_{\{\tilde{X}_{u}^*<x<\tilde X_u+Y_1\}}
\Phi(\tilde{X}_{u}+Y_1-x)\Psi(x-\tilde{X}_{u})\right)-
\mathbb{E}\left({\bf 1}_{\{\tilde{X}_{u}<x<\tilde X_u+Y_1\}}
\Phi(\tilde{X}_{u}+Y_1-x)\Psi(x-\tilde{X}_{u})\right)|
\end{align*}
is less than $||\Phi||_\infty ||\Psi||_\infty \mathbb{P}\left( \tilde{X}_u < x < \tilde{X}_u^*\right)$ which,
by (\ref{fluct}), goes to zero when $u$ goes to zero.
Hence, we obtain   
\begin{align*}
\lim_{h \rightarrow 0} \frac{1}{h} A_2(h)
=\lim_{h \rightarrow 0} \frac{ \lambda e^{-\lambda h}}{h}\int_0^h
\mathbb{E}\left({\bf 1}_{\{\tilde{X}_{u}<x<\tilde X_u+Y_1\}}
\Phi(\tilde{X}_{u}+Y_1-x)\Psi(x-\tilde{X}_{u})\right)du.
\end{align*}  
But, we have the equality of the sets 
\begin{align*}
 \{\tilde{X}_{u}<x{<\tilde X_u+Y_1}\}=
\left\{\{\tilde{X}_{u}<x{<\tilde X_u+Y_1}\}\cap\{Y_1\neq x\}\right\}\cup \left\{ \{\tilde{X}_{u}<x{<\tilde X_u+Y_1}\}\cap\{Y_1= x\}\right\}.
\end{align*}
It follows that 
\begin{align*}
\lim_{h \rightarrow 0} \frac{1}{h} A_2(h)
&=\lim_{h \rightarrow 0} \frac{ \lambda e^{-\lambda h}}{h}\int_0^h
\mathbb{E}\left({\bf 1}_{\left\{\{\tilde{X}_{u}<x{<\tilde X_u+Y_1}\}\cap\{Y_1\neq x\}\right\}}
\Phi(\tilde{X}_{u}+Y_1-x)\Psi(x-\tilde{X}_{u})\right)du\\
&+\lim_{h \rightarrow 0} \frac{ \lambda e^{-\lambda h}}{h}\int_0^h
\mathbb{E}\left({\bf 1}_{\left\{ \{\tilde{X}_{u}<x{<\tilde X_u+Y_1}\}\cap\{Y_1= x\}\right\}}
\Phi(\tilde{X}_{u}+Y_1-x)\Psi(x-\tilde{X}_{u})\right)du.
\end{align*}
Since $\Phi$ and $\Psi$ are Borel and  continuous bounded functions  and   $\tilde{X}$ continuous, Lebesgue's dominated convergence theorem yields
\begin{align*}
\lim_{h \rightarrow 0} \frac{1}{h} A_2(h)&= \lambda \mathbb{E}[\Phi(Y_1-x)\Psi(x){\bf 1}_{\{Y_1>x\}}]
+ \frac{\lambda}{2}\mathbb{E}[\Phi(0)\Psi(Y_1){\bf 1}_{\{Y_1=x\}}].
\end{align*}
\end{proof}
 
\begin{proposition}\label{casaut} 
The joint law of the triplet $(\tau_x, K_x, L_x)$, conditionally on $\{ \tau_x <~\infty\}$,  is defined on $\mathbb{R}_+^* \times \mathbb{R}_+ \times \mathbb{R}_+$
as following 
 \begin{align*}
 p(dt,dk,dl)&=\mathbb{E}(1_{\tau_{x}>T_{N_{t}}}\tilde{f}(t-T_{N_{t}},x-X_{T_{N_{t}}}))\delta_{\{0,0\}}(dk,dl)dt\\
&+{\bf 1}_{\{k\geq 0,l\geq 0\}} \mathbb{E}\left[ {\bf 1}_{\{\tau_{x}>T_{N_t}\}}
\frac{\lambda e^{(-\frac{(x-l-X_{T_{N_t}}-m(t-T_{N_t}))^2}{2(t-T_{N_t})})}} {\sqrt{2\pi (t\!-\!T_{N_t})}}  \right]F_l(dk)dldt\\
&-{\bf 1}_{\{k\geq 0,l\geq 0\}} \mathbb{E}\left[ {\bf 1}_{\{\tau_{x}>T_{N_t}\}}
\frac{\lambda e^{(-\frac{(x-X_{T_{N_t}}+l+m(t-T_{N_t}))^2}{2(t-T_{N_t})}+2m(x-X_{T_{N_t}}))}}{\sqrt{2\pi (t-T_{N_t})}} \right]F_l(dk) dldt\\
\end{align*}
where $F_l(dk)$ is the image of $F_Y(dk)$ by the map $ y \mapsto y-l$.
\end{proposition}

\begin{proof}
 We calculate
\begin{equation}
  \lim_{h \longrightarrow 0} \frac{1}{h×}\mathbb{E}(1_{t<\tau_{x}\leq t+h}\Phi (K_{x})\Psi(L_{x}))
\end{equation}
 for $t>0 $ fixed. For this purpose, we split $\mathbb{E}(1_{t<\tau_{x}\leq t+h}\Phi (K_{x})\Psi(L_{x}))$ according to the values of $N_{t+h}-N_{t}$ as following:
\begin{align}
\label{8} 
 \mathbb{E}(1_{t<\tau_{x}\leq t+h}\Phi (K_{x})\Psi(L_{x}))&= \mathbb{E}(1_{t<\tau_{x}\leq t+h}\Phi (K_{x})\Psi(L_{x})1_{N_{t+h}-N_{t}=0}) \nonumber\\
   & +  \mathbb{E}(1_{t<\tau_{x}\leq t+h}\Phi (K_{x})\Psi(L_{x})1_{N_{t+h}-N_{t}=1})\nonumber\\
   & +  \mathbb{E}(1_{t<\tau_{x}\leq t+h}\Phi (K_{x})\Psi(L_{x})1_{N_{t+h}-N_{t}\geq 2}).   
\end{align}
 So we deal the proof with three steps, the third one being Proposition
\ref{loijointe}

 (i) The third term of the right hand side of (\ref{8}) is upper bounded as following:
\begin{equation*}
 \mathbb{E}(1_{t<\tau_{x}\leq t+h}\Phi (K_{x})\Psi(L_{x})1_{N_{t+h}-N_{t}\geq 2})\leq
 ( 1-e^{ah}-ahe^{ah})\|\Phi\|_\infty \|\Psi\|_\infty.
\end{equation*}
Therefore 
\begin{equation*}
 \lim_{h\longrightarrow 0}\frac{1}{h×}\mathbb{E}(1_{t<\tau_{x}\leq t+h}\Phi (K_{x})\Psi(L_{x})1_{N_{t+h}-N_{t}\geq 2})=0.
\end{equation*}

 (ii) Let us study the first term on the right hand side of (\ref{8}). On the set
\begin{equation*}
 \{ \omega, \quad   N_{t+h}(\omega)-N_{t}(\omega)=~0\},\mbox{ we have }  L_{x}=~0=~K_{x}.  
\end{equation*}
Then
\begin{equation*}
 \mathbb{E}(1_{t<\tau_{x}\leq t+h}\Phi (K_{x})\Psi(L_{x})1_{N_{t+h}-N_{t}=0})=\Phi (0)\Psi (0)\mathbb{E}(1_{t<\tau_{x}\leq t+h}1_{N_{t+h}-N_{t}=0}).
\end{equation*}
Refer to Equation (4) et seq at Section (2.2) in \cite{coutin2011first}, we have
\begin{equation*}
 \lim_{h \longrightarrow 0}\frac{1}{h×}\mathbb{E}(1_{t<\tau_{x}\leq t+h}1_{N_{t+h}-N_{t}=0})=\mathbb{E}(1_{\tau_{x}>T_{N_{t}}}f(t-T_{N_{t}},x-X_{T_{N_{t}}})).
\end{equation*}
So, we  deduce that 
\begin{equation*}
 \lim_{h \longrightarrow 0}\frac{1}{h×} \mathbb{E}(1_{t<\tau_{x}\leq t+h}\Phi (K_{x})\Psi(L_{x})1_{N_{t+h}-N_{t}=0})=\Phi (0)\Psi (0)
\mathbb{E}(1_{\tau_{x}>T_{N_{t}}}f(t-T_{N_{t}},x-X_{T_{N_{t}}})).
\end{equation*}

 The third step is  Proposition \ref{loijointe} which deals with the middle term in (\ref{8}).

\begin{proposition} 
\label{loijointe}  When h goes to $0$,
  $$ h\mapsto \frac{1}{h}\mathbb{E}\left[1_{t<\tau_{x}\leq t+h}\Phi (K_{x})\Psi(L_{x})1_{N_{t+h}-N_{t}=1}\right] $$
  converges to 
\begin{align*}
&  \!\!\int_{0}^{\infty}\!\!\int_{0}^{+\infty}\!\!\Phi (k)\Psi(l) \mathbb{E}\left[ {\bf 1}_{\{\tau_{x}>T_{N_t}\}}
\frac{\lambda e^{(-\frac{(x-l-X_{T_{N_t}}-m(t-T_{N_t}))^2}{2(t-T_{N_t})})}} {\sqrt{2\pi (t\!-\!T_{N_t})}}  \right]F_l(dk)dl\\
&-\int_{0}^{+\infty}\!\!\int_{0}^{+\infty}\!\!  \Phi (k)\Psi(l)
 \mathbb{E}\left[ {\bf 1}_{\{\tau_{x}>T_{N_t}\}}
\frac{\lambda e^{(-\frac{(x-X_{T_{N_t}}+l+m(t-T_{N_t}))^2}{2(t-T_{N_t})}+2m(x-X_{T_{N_t}}))}}{\sqrt{2\pi (t-T_{N_t})}} \right]F_l(dk) dl\\
\end{align*}
where  $F_l(dk)$ is the image of $F_Y(dk)$ by the map $ y \mapsto y-l$.
\end{proposition}
\begin{proof}
We  split the middle term of (\ref{8}) according to the values of $N_t$. Since $\{N_t~=~n\}=~\{ T_n~\leq~t~<~T_{n+1}\}$ so 
\begin{align} \label{split1}
   \mathbb{E}(1_{t<\tau_{x}\leq t+h}\Phi (K_{x})\Psi(L_{x})1_{N_{t+h}-N_{t}=1})& =  \mathbb{E}(1_{t<\tau_{x}\leq t+h}\Phi (K_{x})\Psi(L_{x})
1_{T_{N_{t}}\leq t<T_{N_{t}+1}\leq t+h<T_{N_{t}+2}})\nonumber\\
& = \sum_{n\geq 0} \mathbb{E}(1_{t<\tau_{x}\leq t+h}\Phi (K_{x})\Psi(L_{x})1_{T_{n}\leq t<T_{n+1}\leq t+h<T_{n+2}}).  
\end{align}
We split again the right term of (\ref{split1})   according to
the relative positions between $\tau_x$ and $T_{n+1}$. It follows 
\begin{align} \label{split2}
 \mathbb{E}(1_{t<\tau_{x}\leq t+h}\Phi (K_{x})\Psi(L_{x})1_{N_{t+h}-N_{t}=1})& =
 \sum_{n\geq 0} \mathbb{E}(\Phi (K_{x})\Psi(L_{x}){\bf 1}_{T_{n}\leq t<\tau_{x}<T_{n+1}\leq t+h<T_{n+2}}) \nonumber \\
 &+ \sum_{n\geq 0} \mathbb{E}(\Phi (K_{x})\Psi(L_{x}){\bf 1}_{T_{n}\leq t<\tau_{x}=T_{n+1}\leq t+h<T_{n+2}})\nonumber \\
 &+ \sum_{n\geq 0} \mathbb{E}(\Phi (K_{x})\Psi(L_{x}){\bf 1}_{T_{n}\leq t<T_{n+1}<\tau_{x}\leq t+h<T_{n+2}}).
\end{align}
We recall
\begin{align*}
T_1=S_1,~~ T_2= S_1+S_2,~~~ \cdots,~~~ T_n= \sum_{i=1}^{n} S_i 
\end{align*}
where $\left( S_i \right)_{i \in \mathbb{N}^*} $ is a sequence of independent
 random variables  following an exponential law with parameter $\lambda$.\\
Let be
\begin{align*}
A_h^1&=\sum_{n\geq 0} \mathbb{E}(\Phi (K_{x})\Psi(L_{x}){\bf 1}_{T_{n}\leq t<\tau_{x}<T_{n+1}\leq t+h<T_{n+2}}) \\
A_h^2&= \sum_{n\geq 0} \mathbb{E}(\Phi (K_{x})\Psi(L_{x}){\bf 1}_{T_{n}\leq t<\tau_{x}=T_{n+1}\leq t+h<T_{n+2}})\nonumber \\
A_h^3&= \sum_{n\geq 0} \mathbb{E}(\Phi (K_{x})\Psi(L_{x}){\bf 1}_{T_{n}\leq t<T_{n+1}<\tau_{x}\leq t+h<T_{n+2}}).\nonumber
\end{align*} 
STEP 1: Here, we refer to the analysis of $ \frac{1}{h}B_1(h)$, Equation (4) et seq. in \cite{coutin2011first}.\\
On the sets $ \{T_{n}\leq t< \tau_{x}<T_{n+1}\} $ and $\{T_{n+1}<\tau_{x}\leq t+h<T_{n+2}\}$,
 we have $K_x=0=L_x$. So  
\begin{align*}
A_h^1&= \Phi(0)\Psi(0) \sum_{n\geq 0} \mathbb{E}({\bf 1}_{T_n \leq t<\tau_x<T_{n+1}\leq t+h<T_{n+2}}) \\
&= \Phi(0)\Psi(0) \sum_{n\geq 0} \mathbb{E}({\bf 1}_{T_n \leq t<\tau_x<T_{n}+S_{n+1}\leq t+h<T_{n}+S_{n+1}+S_{n+2}}).\\
\end{align*} 
Strong Markov property at $T_n$ yields
\begin{align*} 
A_h^1= \Phi(0)\Psi(0) \sum_{n\geq 0}\mathbb{E}\left( {\bf 1}_{\{T_n\leq t\}}{\bf 1}_{\{\tau_x >T_n\}}
\mathbb{E}^{T_n}({\bf 1}_{\{t-T_n<\tilde{\tau}_{x-X_{T_n}}<S_{n+1}<t+h-T_n<S_{n+1}+S_{n+2}\}})\right)
\end{align*}
where $\mathbb{E}^{T_n}(.)=\mathbb{E}(.|\mathcal{F}_{T_n})$.
Since $S_{n+2}$ is independent from $S_{n+1},~ \tau_{x-X_{T_n}}$ and  $T_n$, we obtain
\begin{align*}
A_h^1&= \Phi(0)\Psi(0) \sum_{n\geq 0}\mathbb{E}\left( {\bf 1}_{\{T_n\leq t\}}{\bf 1}_{\{\tau_x >T_n\}}
\mathbb{E}^{T_n}(e^{-\lambda(t+h-T_n-S_{n+1})}{\bf 1}_{\{t-T_n<\tilde{\tau}_{x-X_{T_n}}<S_{n+1}<t+h-T_n\}})\right)\\
&\leq e^{-\lambda h} \Phi(0)\Psi(0) \sum_{n\geq 0}\mathbb{E}\left( {\bf 1}_{\{\tau_x >T_n\}}
\mathbb{E}^{T_n}(e^{\lambda S_{n+1}}{\bf 1}_{\{t-T_n<\tilde{\tau}_{x-X_{T_n}}<S_{n+1}<t+h-T_n\}})\right).
\end{align*} 
The random variables $\tilde{\tau}_{x-X_{T_n}}$ and $S_{n+1}$ are independent and their laws admit a density. Therefore
\begin{align*}
\mathbb{E}^{T_n}(e^{\lambda S_{n+1}}{\bf 1}_{\{t-T_n<\tilde{\tau}_{x-X_{T_n}}<S_{n+1}<t+h-T_n\}})&=
\int_{t-T_n}^{t-T_n+h} \tilde{f}(u,x-X_{T_n}) \int_{u}^{t-T_n+h} \lambda ds du\\
&= \int_{t-T_n}^{t-T_n+h}\lambda[t-T_n+h-u] \tilde{f}(u,x-X_{T_n})  du.\\
\end{align*}
The change of variable $v=u-(t-T_n)$ implies  
\begin{align*}
\mathbb{E}^{T_n}(e^{\lambda S_{n+1}}{\bf 1}_{\{t-T_n<\tilde{\tau}_{x-X_{T_n}}<S_{n+1}<t+h-T_n\}})=
\int_0^h \lambda[h-v]\tilde{f}(t-T_n+v,x-X_{T_n})  dv.
\end{align*}
Thus,
\begin{align*}
A_h^1&\leq\lambda e^{-\lambda h} \Phi(0)\Psi(0) \sum_{n\geq 0} \int_0^h \mathbb{E}\left(
   {\bf 1}_{\{\tau_x>T_n\}}[h-v]\tilde{f}(t-T_n+v,x-X_{T_n})\right) dv \\
   &\leq \lambda e^{-\lambda h} \Phi(0)\Psi(0) \int_0^h \mathbb{E}\left(
   {\bf 1}_{\{\tau_x>T_{N_t}\}}[h-v]\tilde{f}(t-T_{N_t}+v,x-X_{T_{N_t}})\right) dv .
\end{align*}
Similarly to the computation of $A_h^1$, we have
\begin{align*}
A_h^3\leq  \lambda e^{-\lambda h} \Phi(0)\Psi(0) \int_0^h \mathbb{E}\left(
   {\bf 1}_{\{\tau_x>T_{N_t}\}}v\tilde{f}(t-T_{N_t}+v,x-X_{T_{N_t}})\right)dv .
\end{align*}
So, 
\begin{align*}
\frac{1}{h}[A_h^1 + A_h^3] \leq\lambda e^{-\lambda h} \Phi(0)\Psi(0) \int_0^h \mathbb{E}\left(
   {\bf 1}_{\{\tau_x>T_{N_t}\}}\tilde{f}(t-T_{N_t}+v,x-X_{T_{N_t}})\right)dv .
\end{align*}
Using the fact that the application  
\begin{align*}
v\mapsto \mathbb{E}\left( {\bf 1}_{\{\tau_x>T_{N_t}\}}\tilde{f}(t-T_{N_t}+v,x-X_{T_{N_t}})\right) \mbox{ is continuous. }
\end{align*}
yields
\begin{align*}
\lim_{h \mapsto 0} \frac{1}{h}[A_h^1 + A_h^3] =0.
\end{align*}
 STEP 2:
We now deal  with $A_h^2$. On the set $\{ \tau_x = T_{n+1}\}$, we have:
\begin{align*}
K_x= X_{\tau_x}-x=X_{T_{n+1}}-x=X_{T_n}+[mS_{1} + W_{S_{1}} +Y_{1}]\circ \Theta_{T_n} -x
\end{align*} 
and 
\begin{align*}
L_x=x-X_{\tau_x^-}=x-X_{T_n}-[mS_{1} + W_{S_{1}}]\circ \Theta_{T_n}
\end{align*}
where $\Theta$ is the shift operator. Since $\tilde{X}_{S_{1}}= [mS_{1} + W_{S_{1}}]$, then
\begin{align*}
K_x=X_{T_n}+ [\tilde{X}_{S_{1}}+ Y_{1}]\circ \Theta_{T_n} -x \mbox{ and }
L_x= x-X_{T_n}-\tilde{X}_{S_{1}}\circ \Theta_{T_n}.
\end{align*}
 So $A_h^2$ can be written as following: $A_h^2=$
 \small{
\begin{align*}
  \sum_{n\geq 0} \mathbb{E}\left({\bf 1}_{\tau_{x}=T_{n}+S_{n+1}>T_n}\Phi (X_{T_n}+\tilde{X}_{S_{n+1}} +Y_{n+1}-x)\Psi(x-X_{T_n}-\tilde{X}_{S_{n+1}})
{\bf 1}_{T_{n}\leq t<T_{n}+S_{n+1}\leq t+h<T_{n}+S_{n+1}+S_{n+2}}\right).
\end{align*} }
Strong Markov property applied at $T_n$ leads to $A_h^2=$
\footnotesize{
\begin{align*}
\sum_{n\geq 0} \mathbb{E}\left({\bf 1}_{\tau_{x}>T_{n}}{\bf 1}_{T_n\leq t}\mathbb{E}^{T_n}(\Phi (X_{T_n}+\tilde{X}_{S_{n+1}}\!\! +Y_{n+1}-x)\Psi(x-\!X_{T_n}-\!\tilde{X}_{S_{n+1}}){\bf 1}_{\tilde{\tau}_{x-X_{T_n}}\!=\!S_{n+1}}
{\bf 1}_{ t-T_{n}<S_{n+1}\leq t+h-T_{n}<S_{n+1}+S_{n+2}})\right).
\end{align*} 
} 
integrating with respect to  $S_{n+2}$ , we have
$e^{\lambda h}A_h^2=$
\footnotesize{ 
$$\!\!\!\sum_{n\geq 0}\!\! \mathbb{E}\left({\bf 1}_{\tau_{x}>T_{n}}{\bf 1}_{T_n\leq t}e^{-\lambda(t-T_n)}\mathbb{E}^{T_n}(e^{\lambda S_{n+1}}\Phi (X_{T_n}\!+\!\tilde{X}_{S_{n+1}}\!\! +\!Y_{n+1}\!-x\!)\Psi(x\!-\!X_{T_n}\!-\!\tilde{X}_{S_{n+1}})\!{\bf 1}_{\tilde{\tau}_{x-\!X_{T_n}}=S_{n+1}}\!{\bf 1}_{ t-T_{n}<S_{n+1}\leq t+h-T_{n}})\right).
$$
} 
We observe that on the set $\{\tilde{\tau}_{x-X_{T_n}}=S_{n+1}\}, Y_{n+1} > 0,~ K_x\geq 0,~  L_x \geq 0   \mbox{ and }
K_x + L_x = Y_{n+1}$. More over: 
\begin{align*}
\{\tilde{\tau}_{x-X_{T_n}}=S_{n+1}\}=\{\sup_{s\leq S_{n+1}}\tilde{X}_s <x-X_{T_n},\tilde{X}_{S_{n+1}}+Y_{n+1}>x-X_{T_n}\}.
\end{align*}
Integrating with respect to $(S_{n+1}, Y_{n+1})$ implies that
\begin{align*}
&\mathbb{E}^{T_n}(e^{\lambda S_{n+1}}\Phi (X_{T_n}\!+\!\tilde{X}_{S_{n+1}}\! +\!Y_{n+1}-x)\Psi(x-X_{T_n}\!-\!\tilde{X}_{S_{n+1}}){\bf 1}_\{\tilde{\tau}_{x-X_{T_n}}=S_{n+1} \}{\bf 1}_{ t-T_{n}<S_{n+1}\leq t+h-T_{n}})=\\
&\int_0^{+\infty} \int_{t-T_n}^{t+h-T_n}\lambda\mathbb{E}^{T_n}[\Phi(X_{T_n}\!+\!\tilde{X}_{u} +y-x)\Psi(x-X_{T_n}-\tilde{X}_{u}){\bf 1}_\{\sup_{s\leq u}\tilde{X}_s <x-X_{T_n},\tilde{X}_{u}+y>x-X_{T_n}\}] F_Y(dy)du.
\end{align*}
According to Corollary 3.2.1.2  page 147 of \cite{jeanblanc2009mathematical}, $(\sup_{s\leq u}\tilde{X}_s,\tilde{X}_u)$ admits a density
\begin{align*}
 \tilde{p}(b,a,t)=\frac{2(2b-a)}{\sqrt{2\pi t^3}} \exp\left[-\frac{(2b-a)^2}{2t}+ma-\frac{m^2}{2}t\right] {\bf 1}_{b>\max\{0,a\}}.
\end{align*}
So, $A_h^2$ is equal to 
\scriptsize{
\begin{align*}
e^{-\lambda h}\sum_{n\geq 0}\mathbb{E}\left[{\bf 1}_{\tau_{x}>T_{n}}{\bf 1}_{T_n\leq t}e^{-\lambda(t-T_n)} 
\!\!\int_{\mathbb{R}^2}\!\!\int_0^{+\infty}\!\!\! \int_{t-T_{n}}^{t+h-T_{n}}\!\!\!\!\lambda\Phi (X_{T_{n}}+a +y-x)\Psi(x\!-\!X_{T_{n}}\!-\!a){\bf 1}_\{b\!<\!x\!-\!X_{T_{n}},a+y\!>\!x\!-\!X_{T_{n}}\} \tilde{p}(b,a,u)\right]\\
 F_Y(dy)dudbda.
\end{align*}
}
Since $e^{-\lambda(t-T_n)}=\mathbb{E}^{T_n}({\bf 1}_{T_{n+1}>t}),$ and on the event $\{ N_t =n\},~~ T_{N_t+1}>t~~ a.s,$ we have
\small{
\begin{align*}
A_h^2=e^{-\lambda h}\mathbb{E}\left[{\bf 1}_{\tau_{x}>T_{N_t}}
\!\!\int_{\mathbb{R}^2}\!\!\int_0^{+\infty}\!\!\! \int_{t-T_{N_t}}^{t+h-T_{N_t}}\!\!\!\!\lambda\Phi (X_{T_{N_t}}+a +y-x)\Psi(x\!-\!X_{T_{N_t}}\!-\!a){\bf 1}_{\{b\!<\!x\!-\!X_{T_{N_t}},a+y\!>\!x\!-\!X_{T_{N_t}}\}} \tilde{p}(b,a,u)\right]\\
 F_Y(dy)dudbda.
\end{align*}
}
We compute the integral with respect to $db$ and it follows 
\footnotesize{
\begin{align*}
& A_h^2=e^{-\lambda h}\mathbb{E}\left[{\bf 1}_{\tau_{x}>T_{N_t}}
\!\!\int_{x-y-X_{T_{N_t}}}^{x-X_{T_{N_t}}}\!\!\int_0^{+\infty}\!\! \int_{t-T_{N_t}}^{t+h-T_{N_t}}\!\!\!\!\Phi (X_{T_{N_t}}+a +y-x)\Psi(x-X_{T_{N_t}}-a)
\frac{\lambda e^{-\frac{(a-mu)^2}{2u}}}{\sqrt{2\pi u}}F_Y(dy)du da\right]-\\
&e^{-\lambda h}  \mathbb{E}\left[{\bf 1}_{\tau_{x}>T_{N_t}}
\!\!\int_{x-y-X_{T_{N_t}}}^{x-X_{T_{N_t}}}\!\!\int_0^{+\infty}\!\! \int_{t-T_{N_t}}^{t+h-T_{N_t}}\!\!\!\!\Phi (X_{T_{N_t}}\!+a +y\!-\!x)\Psi(x\!-\!X_{T_{N_t}}\!-\!\!a)
\frac{\lambda e^{(-\frac{(a-mu-2x+\!2X_{T_{N_t}}\!)^2}{2u}\!+\!2m(x-\!X_{T_{N_t}}\!))}}{\sqrt{2\pi u}}F_Y(dy)du da\right].
\end{align*}   } 
Change of variables $v=u-(t-T_{N_t})$  and $l= x-X_{T_{N_t}}-a $  yields
\begin{align*}
& A_h^2=e^{-\lambda h}  
\mathbb{E}\left[\!\!\int_{0}^{y}\!\!\int_{0}^{+\infty}\!\! \int_{0}^{h}
{\bf 1}_{\{\tau_{x}>T_{N_t}\}}\Phi (y-l)\Psi(l)
\frac{\lambda e^{(-\frac{(x-l-X_{T_{N_t}}-m(t-T_{N_t}+v))^2}{2(t-T_{N_t}+v)})}} {\sqrt{2\pi (t\!-\!T_{N_t}+v)}}dvF_Y(dy) dl \right]\\
&-e^{-\lambda h}  
 \mathbb{E}\left[\!\!\int_{0}^{y}\!\!\int_{0}^{+\infty}\!\! \int_{0}^{h}
 {\bf 1}_{\{\tau_{x}>T_{N_t}\}} \Phi (y-l)\Psi(l)
\frac{\lambda e^{(-\frac{(x-X_{T_{N_t}}+l+m(t-T_{N_t}+v))^2}{2(t-T_{N_t}+v)}+2m(x-X_{T_{N_t}}))}}{\sqrt{2\pi (t-T_{N_t}+v)}}dvF_Y(dy) dl \right].\\
\end{align*}
Let $F_l$ be the image of $F_Y$ by the map $y\mapsto y-l$.
Hence,
\begin{align*}
\lim_{h \mapsto 0}& \frac{1}{h} A_h^2=\int_{0}^{+\infty}\!\!\int_{0}^{+\infty}\!\!\Phi (k)\Psi(l) \mathbb{E}\left[ {\bf 1}_{\{\tau_{x}>T_{N_t}\}}
\frac{\lambda e^{(-\frac{(x-l-X_{T_{N_t}}-m(t-T_{N_t}))^2}{2(t-T_{N_t})})}} {\sqrt{2\pi (t\!-\!T_{N_t})}}  \right]F_l(dk)dl\\
&- \int_{0}^{+\infty}\!\!\int_{0}^{+\infty}\!\!  \Phi (k)\Psi(l)
 \mathbb{E}\left[ {\bf 1}_{\{\tau_{x}>T_{N_t}\}}
\frac{\lambda e^{(-\frac{(x-X_{T_{N_t}}+l+m(t-T_{N_t}+))^2}{2(t-T_{N_t})}+2m(x-X_{T_{N_t}}))}}{\sqrt{2\pi (t-T_{N_t})}} \right]F_l(dk) dl\\
\end{align*}
since
\begin{align*}
v\mapsto 
 &\int_{0}^{+\infty}\!\!\int_{0}^{+\infty}\!\!\Phi (k)\Psi(l) \mathbb{E}\left[ {\bf 1}_{\{\tau_{x}>T_{N_t}\}}
\frac{\lambda e^{(-\frac{(x-l-X_{T_{N_t}}-m(t-T_{N_t}+v))^2}{2(t-T_{N_t}+v)})}} {\sqrt{2\pi (t\!-\!T_{N_t}+v)}}  \right]F_l(dk)dl\\
&-\int_{0}^{+\infty}\!\!\int_{0}^{+\infty}\!\!  \Phi (k)\Psi(l)
 \mathbb{E}\left[ {\bf 1}_{\{\tau_{x}>T_{N_t}\}}
\frac{\lambda e^{(-\frac{(x-X_{T_{N_t}}-l+l+m(t-T_{N_t}+v))^2}{2(t-T_{N_t}+v)}+2m(x-X_{T_{N_t}}))}}{\sqrt{2\pi (t-T_{N_t}+v)}} \right]F_l(dk) dl\\
\end{align*}
is continuous. 
\end{proof}
This proposition concludes the proof of Proposition \ref{casaut} .
\end{proof}
\section{Conclusion} 
Our study relies on the default time of a L\'evy process.
We have first shown that the distribution function of the default time $\tau_x$
 belongs to $\mathcal{C}(\R^*_+\times\R^*_+)$ 
 and for any $x\in\R^*_+,$ to  $\mathcal{C}(\R_+)$.
 {This will be very useful in our future works on default time of a L\'evy process where we will use the filtering theory.
 Secondly,} we have obtained an explicit expression to characterize the joint law of the hitting time, overshoot 
 and undershoot of one L\'evy process. In this expression, the Gaussian density is of great importance.
  This law gives a lot of information on the deficit and surplus at default time. 
In a following paper in progress, we will give a partial differential equation for a L\'evy process
 and its running  maximum.

\section{Appendix} 
\begin{lemme}\label{lem-A-1}
Let $\beta >-1,$ then for $0 <t \leq 1,$ 
\begin{align*}
{\mathbb E} \left( {\mathbf 1}_{\{ N_t \geq 2\}}(t-T_{N_t})^{\beta}\right) \leq \left( 
\sum_{n=1}^{\infty}\frac{\lambda^ne^t}{(n-1)!}  B( n, \beta+1)\right) t^{2+\beta }
\end{align*} 
\end{lemme}
where  $B( n, \beta+1)=\int_0^1 (1-u)^\beta u^{n-1} du.$ 
\begin{proof}
According to the  values of the process $N_t$, we have
\begin{align*}
{\mathbb E} \left(  {\mathbf 1}_{\{ N_t \geq 2\}}(t-T_{N_t})^{\beta}\right) &= \sum_{ n\geq 2} \mathbb{E}\left((t-T_n)^{\beta}{\bf 1}_{T_n\leq t <T_n + S_{n+1}}\right)\\
& = \sum_{ n\geq 2} \mathbb{E}\left(e^{-\lambda(t-T_n)}(t-T_n)^{\beta}\right).
\end{align*}
Since $T_n$ follows a Gamma law of parameters $n$ and $\lambda$, hence
$$
{\mathbb E} \left(  {\mathbf 1}_{\{ N_t \geq 2\}}(t-T_{N_t})^{\beta}\right)=
 \sum_{n=2}^{\infty}\frac{\lambda^ne^{-\lambda t}}{(n-1)!} t^{n+ \beta}B( n, \beta+1).
$$
\end{proof}
\begin{lemme}\label{ui}
For any $0<t_1 <t_2$,
\begin{equation}\label{majorlemme} 
\sup_{t_1\leq t \leq t_2}\mathbb{E}\left(1_{\tau_{x}>T_{N_{t}}}\tilde{f}(t-T_{N_{t}}, x-X_{T_{N_{t}}})^p\right)< +\infty.
\end{equation}
and for instance, as a consequence, the family 
\begin{equation*}
\left(1_{\tau_{x}>T_{N_{t}}}\tilde{f}(t-T_{N_{t}}, x-X_{T_{N_{t}}}),~~t\in [t_1,t_2],~~ x>0\right)
\end{equation*}
is uniformly integrable.
\end{lemme}  
\begin{proof}
Refer to Lemma 3.1 in \cite{coutin2011first}, if $ G $ is a Gaussian random variable $\mathcal{N}(0,1), \mu~>~0$, $u > 0, \sigma \in~\mathbb{R}^{+}
\mbox{ and } p\geq~1 ,$ then
{\small
\begin{equation*}
\mathbb{E}[(\tilde{f}(u, \mu+\sigma G))^p1_{\mu+\sigma G >0}]=\frac{1}{\sqrt{(2\pi)^{p}}} \frac{u^{\frac{1-2p}{2}} e^{-
\frac{p(\mu-mu)^{2}}{2(p\sigma^{2}+u)}×}}
{(p\sigma^{2}+u)^{\frac{p+1}{2×}}×}\mathbb{E}\left[\left(\sigma G+ \sqrt{\frac{u}{p\sigma^{2}+u×}}(\mu-mu)+
m\sqrt{u(p\sigma^{2}+u})\right)_{+}^{p}\right]
\end{equation*}
}

Using the inequality $ (a+ b + c)^{p} \leq 3^{p-1}(a^{p}+b^{p}+c^{p}) $ where a, b, c are positive numbers, it follows that 
{\small
\begin{equation*}
 \mathbb{E}[\tilde{f}(u, \mu+\sigma G)^{p}1_{\mu+\sigma G >0}]\leq \frac{3^{p-1}}{\sqrt{(2\pi)^{p}}} \frac{u^{\frac{1-2p}{2}}
e^{-\frac{p(\mu-mu)^{2}}{2(p\sigma^{2}+u)}×}}{(p\sigma^{2}+u)^{\frac{p+1}{2×}}×}\left( \sigma^{p}\mathbb{E}(\mid G \mid^{p})+
(\frac{u}{p\sigma^{2}+u})^{\frac{p}{2}} (\mu - mu)^{p}+ m^{p}(u(p\sigma^{2}+u))^{\frac{p}{2}}\right).
\end{equation*}
} 
Therefore
\begin{align*} 
    \mathbb{E}[\tilde{f}(u, \mu+\sigma G)^{p}1_{\mu+\sigma G >0}] &\leq  \frac{3^{p-1}}{\sqrt{2\pi}^{p}}
    \left(\frac{u^{\frac{1-2p}{2}}}{(p\sigma^{2}+u)^{\frac{p+1}{2}}}\sigma^p\mathbb{E}(| G |^p)+
    \frac{u^{\frac{1-p}{2}}}{(p\sigma^{2}+u)^{\frac{p+1}{2}}}c_{p} +
    \mid m \mid^{p}\frac{u^{\frac{1-p}{2}}}{(p\sigma^{2}+u)^{\frac{1}{2}}}\right)
    \\
    & \leq  \frac{3^{p-1}}{\sqrt{2\pi}^{p}} \left(\frac{u^{\frac{1-2p}{2}}}{(p\sigma^{2}+u)^{\frac{p+1}{2}}}\sigma^{p}
    \mathbb{E}(| G |^p)+
    \frac{c_{p}}{×u^{\frac{p-1}{2×}(p\sigma^{2}+u)^{\frac{p+1}{2×}}}}+ 
    \frac{|m|^p}{×u^{\frac{p-1}{2×}}(p\sigma^{2}+u)^{\frac{1}{2×}}}\right).
\end{align*}
Using the independence between the Brownian motion and the Poisson process, 
we can apply this  inequality to $\sigma=\sqrt{T_{N_{t}}}, u=t-T_{N_{t}},
 p\sigma^{2}+u=~(p-~1)T_{N_{t}}~+~t~\geq~t~>~t_1 $:
\begin{align*}
& \mid \mathbb{E}( 1_{\tau_{x}>T_{N_{t}}}\tilde{f}(t-T_{N_{t}}, x-X_{N_{t}})^{p}) \mid \leq 
 \\
& \frac{3^{p-1}}{\sqrt{2\pi}^p}\mathbb{E}\left(\frac{T_{N_{t}}^{\frac{p}{2}}\mathbb{E}(\mid G \mid^{p})}{(t-T_{N_{t}})^{\frac{2p-1}{2}} t^{\frac{p+1}{2}}}
 + \frac{c_{p}}{(t-T_{N_{t}})^{\frac{p-1}{2}}t^{\frac{2p+1}{2}}}
 + \frac{\mid m \mid^{p}}{\sqrt{t}(t-T_{N_{t}})^{\frac{p-1}{2}}} \right).
\end{align*}
We use the following  for $\alpha=0$ or $p/2,$ and $\beta=(2p-1)/2$ or $(p-1)/2$: 
\begin{align*} 
 \mathbb{E}(\frac{T_{N_{t}}^{\alpha}}{(t-T_{N_{t}})^{\beta}}) & = \frac{1}{t^\beta}+
 \sum_{n\geq 1}\mathbb{E}[\frac{T_{n}^{\alpha}}{(t-T_{n})^{\beta}}1_{T_{n}<t<T_{n+1}}]\\
& = \frac{1}{t^\beta}+ \sum_{n\geq 1}\int_{0}^{t}\frac{u^{\alpha}}{(t-u)^{\beta}}\frac{(\lambda u)^{n-1}}{(n-1)!×}\lambda e^{-\lambda u}
\int_{t-u}^{+\infty}\lambda e^{-\lambda v}dvdu\\
& = \frac{1}{t^\beta}+ e^{-\lambda (t)}\sum_{n\geq 1} \frac{\lambda^{n}}{(n-1)!×}
\int_{0}^{t}\frac{u^{n+\alpha-1}}{(t-u)^{\beta}}du\\
& = \frac{1}{t^\beta}+ e^{-\lambda (t)}\sum_{n\geq 1}\frac{[\lambda(t)]^{n}}{(n-1)!×}B(\alpha + n; 1-\beta)(t)^{\alpha-\beta}
\end{align*}
where $B( n, \beta+1)=\int_0^1 (1-u)^\beta u^{n-1} du.$
Since $t_2\geq t \geq t_1>0$ , we conclude.
\end{proof}

\begin{lemme}\label{lem-A-0} For any numbers $u,~v,~y,~z$ and $a$, the equality
\begin{align*}
\frac{1}{u} ( a-z)^2 +\frac{1}{v} (a-y)^2= \frac{v+u}{uv}\left[ a -\frac{vz+uy}{v+u} \right]^2 +\frac{1}{u+v}(z-y)^2
\end{align*}
holds.
\end{lemme}
\begin{proof}
We develop both squared
\begin{align*}
\frac{1}{u} ( a-z)^2 +\frac{1}{v} (a-y)^2= (\frac{1}{u} + \frac{1}{v})a^2 - 2a(\frac{z}{u} +\frac{y}{v}) + \frac{z^2}{u} +\frac{y^2}{v}.
\end{align*}
We have the first square
\begin{align*}
\frac{1}{u} ( a-z)^2 +\frac{1}{v} (a-y)^2&= (\frac{1}{u} + \frac{1}{v})\left[ a^2 - 2a(\frac{z}{u} +\frac{y}{v})\frac{uv}{v+u} +\left[(\frac{z}{u} +\frac{y}{v})\frac{uv}{v+u} \right]^2 \right]\\
 &+ \frac{z^2}{u} +\frac{y^2}{v}- \left[ \frac{z^2}{u^2} + 2\frac{yz}{uv} +\frac{y^2}{v^2} \right] \frac{uv}{u+v}.
\end{align*}
We order
\begin{align*}
\frac{1}{u} ( a-z)^2 +\frac{1}{v} (a-y)^2&= (\frac{1}{u} + \frac{1}{v})\left[ a^2 - 2a(\frac{z}{u} +\frac{y}{v})\frac{uv}{v+u} +\left[(\frac{z}{u} +\frac{y}{v})\frac{uv}{v+u} \right]^2 \right]\\
 &+ z^2\frac{1}{u}[ 1- \frac{v}{u+v}] + y^2 \frac{1}{v}[ 1- \frac{u}{u+v}]  -  2\frac{yz}{uv} \frac{uv}{u+v}.
\end{align*}
which concludes the proof. 
\end{proof}
\begin{lemme}\label{lem-A-3} 
{Let be $(\tilde{X}_u, u\geq 0)$ be a Brownian motion with drift $m\in \mathbb{R}$. So, we have}
\begin{align*}
{\mathbb E} \left( {\mathbf 1}_{\{\tilde{X}^*_u < c\}} | \tilde{X_u} \right)= {\mathbf 1}_{\{ \tilde{X}_u <c\} }\left[ 1- \exp \left[- \frac{2 c^2}{u} + \frac{2c }{u} \tilde{X}_u \right] \right]
\end{align*}
for all real number $u> 0.$
\end{lemme}
\begin{proof}
Refer to Corollary 3.2.1.2  page 147 of \cite{jeanblanc2009mathematical}, $(\tilde{X}_t^*,\tilde{X}_t)$ admits a density
\begin{align*}
 \tilde{p}(b,a,t)=\frac{2(2b-a)}{\sqrt{2\pi t^3}} \exp\left[-\frac{(2b-a)^2}{2t}+ma-\frac{m^2}{2}t\right] {\bf 1}_{b>\max\{0,a\}}
\end{align*}
So, the conditional law of $\tilde{X}_u^*$ given $\tilde{X}_u =a$  has the density
\begin{align*}
 f_{\tilde{X}^*|\tilde{X}=a}(b|a)= \frac{2(2b-a)}{u}\exp\left[-\frac{(2b-a)^2-a^2}{2u}\right]{\bf 1}_{b>\max\{0,a\}}.
\end{align*}
Thus
\begin{align*}
{\mathbb E} \left( {\mathbf 1}_{\{\tilde{X}^*_u < c\}} | \tilde{X_u}=a \right)&= {\{\bf 1}_{a<c\}}
\int_{\max\{a,0\}}^c \frac{2(2b-a)}{u}\exp\left[-\frac{(2b-a)^2-a^2}{2u}\right]{\bf 1}_{b>\max\{0,a\}} db \\
&= {\mathbf 1}_{\{a <c\} }\left[ 1- \exp \left[- \frac{2c^2}{u} + \frac{2c }{u} a \right] \right].
\end{align*}
\end{proof}
\noindent
{\bf Acknowledgments} We wish to thank Monique Pontier for her carefully reading this paper and making useful suggestions.

\end{document}